\documentclass[a4paper,11pt]{amsart} 

\usepackage{amssymb, amsmath, amsthm}
\usepackage[utf8]{inputenc}
\usepackage{graphicx}
\usepackage[final]{hyperref}
\usepackage{color}
\usepackage[normalem]{ulem}
\usepackage{epstopdf}

\usepackage[a4paper, centering]{geometry}
\usepackage{color}
\usepackage{graphicx}
\usepackage{scalerel,stackengine}
\usepackage{mathtools}
\usepackage{esint}
\usepackage{mathrsfs}  

\usepackage{mathtools}
\usepackage{enumitem}%

\newcommand \Ln{x}
\newcommand \Per{\text{Per}}

\newcommand \Ls{\mathcal{H}^{n-1}}

\newcommand \Rn{\mathbb{R}^n}

\newcommand \Ece{\mathcal{E}^{c,\epsilon}}
\newcommand \Ecek{\mathcal{E}_{k}^{c,\epsilon}}
\newcommand \Ec{\mathcal{E}^{c}}

\newcommand \Um{\mathcal{U}_{m}}
\newcommand \goto{\underset{\epsilon\rightarrow 0}{\longrightarrow}}

\newcommand \vps{{\varepsilon}}
\newcommand{\Om}{\Omega}

\newcommand{\lb}{\lambda}

\newcommand{\sm}{\setminus}

\newcommand{\sq}{\subseteq}
\newcommand{\ov}{\overline}

\def\RR{\mathbb{R}}
\def\ds{\displaystyle}

\let\oldsqrt\sqrt
\def\sqrt{\mathpalette\DHLhksqrt}
\def\DHLhksqrt#1#2{%
\setbox0=\hbox{$#1\oldsqrt{#2\,}$}\dimen0=\ht0
\advance\dimen0-0.2\ht0
\setbox2=\hbox{\vrule height\ht0 depth -\dimen0}%
{\box0\lower0.4pt\box2}}

\newtheorem{theorem}{Theorem}
\newtheorem{proposition}[theorem]{Proposition}
\newtheorem{lemma}[theorem]{Lemma}

\theoremstyle{definition}
\newtheorem{definition}[theorem]{Definition}
\newtheorem{remark}[theorem]{Remark}

\numberwithin{equation}{section}

\usepackage[colored]{shadethm}

\begin{document}

\title[Spectral inequalities in quantitative form]{Degenerate free discontinuity problems and spectral inequalities in quantitative form}
\author[D. Bucur]
{Dorin Bucur}
\address[Dorin Bucur]{Univ. Savoie Mont Blanc, CNRS, LAMA \\
73000 Chamb\'ery, France
}
\email[D. Bucur]{  dorin.bucur@univ-savoie.fr}
\author[A. Giacomini]
{Alessandro Giacomini}
\address[Alessandro Giacomini]{DICATAM, Sezione di Matematica, Universit\`a degli Studi di Brescia, Via Branze 43, 25133 Brescia, Italy}
\email[A. Giacomini]{alessandro.giacomini@unibs.it}

\author[M. Nahon]
{Mickaël Nahon}
\address[Mickaël Nahon]{Univ. Savoie Mont Blanc, CNRS, LAMA \\
73000 Chamb\'ery, France
}
\email[M. Nahon]{  mickael.nahon@univ-smb.fr}

\thanks{A.G. is member of the Gruppo Nazionale per L'Analisi Matematica, la Probabilit\`a e loro Applicazioni (GNAMPA) of the Istituto Nazionale di Alta Matematica (INdAM)}




\begin{abstract} 
We introduce a new geometric-analytic functional that we analyse in the context of free discontinuity problems. Its main feature is that the geometric term (the length of the jump set) appears with negative sign. This is motivated by 
searching quantitative inequalities for   best constants of Sobolev-Poincar\'e inequalities with trace terms in $\RR^n$ which correspond to fundamental eigenvalues associated to semilinear problems for the Laplace operator with Robin boundary conditions. Our method is based on the study of this new, degenerate,  functional which involves an obstacle problem in interaction with the jump set. Ultimately, this becomes a mixed free discontinuity/free boundary problem occuring above/at the level of the obstacle, respectively.    \end{abstract}

\date{\today}
\maketitle
\tableofcontents
\section{Introduction}
Free discontinuity problems emerged in the context of the analysis of the Mumford-Shah functional, later on around different crack propagation models of Francfort and Marigo type and more recently around  shape optimization problems of Robin type. The common feature of all those problems is, roughly speaking, the minimization of a sum between an energy term corresponding to a certain state equation issued from the model and of some more geometric terms involving the volume of the domain of the PDE, the length of the jump set or some more complex jump energy. A formal example could be written as
$$\min \{E(u) + {\mathcal H}^{n-1}(J_u) : u \in SBV_{loc} (\RR^n)\}.$$
 The balance between the energy $E(\cdot)$ of the PDE and the geometric term (above the length of the jump set) is the key phenomenon leading to a solution of the free discontinuity problem. 

The main focus of this paper is to introduce and analyse a new  analytic-geometric functional involving both an energy of a PDE and the length of the jump set, in which the geometric term appears with negative sign. The exact description is given in the next section but, formally, this could be written as
$$\min \{E(u) - {\mathcal H}^{n-1}(J_u) : u \in SBV_{loc} (\RR^n)\}.$$
Of course, at a first sight this may appear surprising! Presumably, the negative sign would lead to non-existence of a solution and ill posedness. However, this is not always the case, as the presence of the jump energy with a negative sign can sometimes be balanced by the energy of the PDE. As we will show in the next section, this is the case if the jump set acts as an obstacle and the energy contains some mass of the state function on the jump set. Robin boundary conditions can be suitably adapted to play this role.  Ultimately, this leads to  a new (degenerate) problem which takes the form of a free discontinuity problem above the obstacle and of a free boundary problem at the level of  the obstacle. 

This kind of problems   pops   up naturally in the context of searching quantitative forms of spectral isoperimetric inequalities for  eigenvalues of nonlinear Robin Laplacian problems, in which the ball is expected to be a solution. Proving that the minimizer of the associated analytic-geometric functional is the ball, gives straight away a spectral isoperimetric inequality in a quantitative form.

In order to introduce the functional, we recall our objectives.

\medskip

 \noindent{\bf The context of quantitative isoperimetric inequalities.} The sharp quantitative isoperimetric inequality proved by Fusco, Maggi and Pratelli in 2008 (see \cite{FMP08}) reads

\begin{equation}\label{bgn01}
|\Omega|^{\frac{1-n}{n}}\Per(\Omega) - |B|^{\frac{1-n}{n}}\Per(B)\ge C(n)\,\mathcal A(\Omega)^2,
\end{equation}
where $\Om \sq \RR^n$ is a measurable set, $B$ is a ball of the same volume as $\Om$, $\Per (\Om)$ is the generalized perimeter of $\Om$ and
$$\mathcal A(\Omega)= \inf\left\{\frac{|\Omega \triangle  B |}{|\Omega|} :  B\subset \RR^n, | B|=|\Om|\right\},$$
is the  Fraenkel asymmetry. 

In the vein of this inequality, in the last decade  intensive research was carried to obtain quantitative versions of some classical spectral inequalities, like Faber-Krahn, Szeg\"o-Weinberger, Saint-Venant, Weinstock and many others. We refer the reader to the recent survey by Brasco and De Philippis \cite{BDP17} for an overview of the topic.

In \cite{BP12}, Brasco and Pratelli prove a sharp quantitative form for the Szeg\"o-Weinberger inequality 
$$|B|^{2/N}\mu_1(B)-|\Omega|^{2/n}\mu_1(\Omega)\ge C(n)\,\mathcal A(\Omega)^2,
$$
and in  \cite{BDPR12} Brasco, De Philippis and Ruffini found a similar quantitative form of the Brock-Weinstock inequality. The common feature of both results is that the ball corresponds to a {\it maximal} value. Loosely speaking, the strategy to prove such an inequality relies on studying some weighted form of \eqref{bgn01}, via a suitable choice of test functions. 

Spectral inequalities where the ball is {\it minimal}, like the Faber-Krahn inequality for the Dirichlet Laplacian, requires  a completely different approach, since the use of fixed test functions is not anymore useful. The first results on the quantitative form of the Faber-Krahn inequalities were obtained by  Melas \cite{Me92} and Hansen and Nadirshvili \cite{HN94} for simply connected sets in dimension $2$ and convex sets in $\RR^n$, but the complete proof of the sharp form of the quantitative inequality was given only in 2015 by Brasco, De Philippis and Velichkov \cite{bra15}. A fundamental idea in their proof is to use a selection principle, in the spirit of Cicalese and Leonardi \cite{CL12}, which, roughly speaking, reduces the class of sets $\Om$ for which the inequality has to be proved to a much smaller one, consisting on smooth, small graph perturbations of the ball which can be handled by local perturbation arguments. The selection of those sets is done by solving a suitable {\it auxiliary free boundary problem}; this part concentrates the most of the technicalities. Following the same strategy, nonlinear eigenvalues   were  discussed by Fusco and Zhang in \cite{FZ17}.

The purpose of this paper is to get quantitative isoperimetric inequalities for the best constants of  Sobolev-Poincar\'e inequalities with trace terms.
 Those constants are fundamental semilinear eigenvalues of the Laplace operator with Robin boundary conditions and can be expressed by minimization of suitable Rayleigh quotients. Our objective could be compared to the quantitative inequalities of Faber-Krahn type obtained for Dirichlet boundary conditions in \cite{bra15} and \cite{FZ17}, but from a technical point of view the solution is completely different. 

\medskip
\noindent{\bf Quantitative spectral inequalities for the Robin Laplacian.}
Let $\beta >0$. For every bounded, open Lipschitz set $\Om\sq \RR^n$ and for every $q \in [1, \frac{2n}{n-1})$ one defines
\begin{equation}\label{bgn07}
\lb_q(\Om) =\inf_{u\in H^1(\Om),  u\not=0} \frac{\ds \int_{\Omega}|\nabla u|^2 d\Ln+\int_{\partial\Omega}\beta u^2d\Ls}{\left(\ds \int_{\Omega}|u|^q d\Ln\right)^{2/q}}.
\end{equation}
Our objective is to prove that 
\begin{equation}\label{bgn06}
\lb_q(\Om)-\lb_q(B)\ge C\,\mathcal A(\Omega)^2
\end{equation}
where the constant $C>0$ depends on $n, \beta, q$ and $|\Om|$, but not on $\Om$.  

\smallskip
\noindent {\it The non-quantitative version of \eqref{bgn06}: the case $C=0$.}  Before proving \eqref{bgn06} in its quantitative form, with $C>0$, it is convenient to recall that the inequality is true with $C=0$ for every $q\in [1,2]$. The minimality of the ball among all Lipschitz sets of the same volume for the first Robin eigenvalue of the Laplacian (i.e. $q=2$)
\begin{equation}\label{bgn02}
\lb_2(\Om)-\lb_2(B)\ge 0
\end{equation}
 was proved in  two steps,  by Bossel in $\RR^2$ in 1986 (see \cite{Bo86})  and  by Daners in $\RR^n$ in 2006 (see \cite{Da06}).  The proofs are quite involved  and definitely require new ideas with respect to the Faber-Krahn inequality, namely the  analysis of the so called {\it $H$-function}. As we do not use this function here and because it is quite technical, we shall not detail it here (the reader is referred to  \cite{Da06}). Nevertheless, it is important to say that intensive efforts were done to build similar \mbox{$H$-functions} for other values of $q\not=2$, in particular for the special case $q=1$ corresponding to the torsional rigidity, with the objective to extend the Saint-Venant inequality. Up to now,  they were not successful  and it is likely that such an $H$-function may not exist, so that a proof similar to Bossel-Daners in the case $q\not=2$ cannot be produced. However, the inequality $\lb_q(\Om)- \lb_q(B)\ge 0$ has been proved, using a different strategy, for any $q\in [1,2]$ in \cite{Buc15} (see also \cite{Buc10}), while for  $q\in (2,\frac{2n}{n-1})$ it has been proved in a slightly weaker form. The proof is based on a free discontinuity approach in which the inequality is seen as a minimization problem in the class of special functions of bounded variation in $\RR^n$ (see Section \ref{bgn04} below). 

\smallskip
\noindent {\it The quantitative version of \eqref{bgn06}: the case $C>0$.} Coming back to the quantitative form \eqref{bgn06},  in \cite{BFNT18} the result was proved for $q=2$, {\it only}. The reason was of technical nature. Precisely, the proof makes crucial use of the $H$-function, available only for $q=2$.
Indeed, there are two steps in the proof  (\cite{BFNT18}): the first step is based on a deeper analysis of the H-function of Bossel and Daners which led to the  {\it  intermediate} inequality 
\begin{equation}\label{bgn05}
\lb_2(\Om) - \lb_2(B)\ge \frac{\beta}{2} \inf_{x\in \Om} u^2(x) (\Per(\Om)-\Per(B)).
\end{equation}
By itself, this inequality is interesting and already quantitative, but not uniform, as the difference of the perimeters on the right hand side is multiplied by the infimum of an $L^2$-normalized eigenfunction $u$, which depends on $\Om$.  
In a second step, one uses the selection principle to replace $\Om$ by a new set which, roughly speaking, has lower eigenvalue, comparable Fraenkel asymmetry and a controlled, uniform, lower bound of the eigenfunction. The new set is build as a minimizer of a suitable 
 {\it auxiliary free discontinuity problem}. This last step recalls both the strategy of Cicalese and Leonardi for the quantitative isoperimetric inequality and the one of Brasco, De Philippis and Velichkov for the first Dirichlet eigenvalue. The difference is however fundamental, as one has to solve a free discontinuity problem with a completely different objective. Indeed, one aims to compare a general set with another set, with a comparable Fraenkel asymmetry and lower eigenvalue, for which the lower bound of the state function is controlled from below (in order to use the intermediate inequality of Step 1). Meanwhile, in \cite{CL12}
and \cite{bra15} the solutions of the associated free boundary problems had the objective  to compare a general set with a set which is graph over the ball (in order to use  second order differential inequalities).

 \medskip
 \noindent{\bf A new functional to handle  quantitative inequalities.}  
The main purpose of the paper is to obtain the quantitative inequality \eqref{bgn06}. While the selection principle in association with the auxiliary free discontinuity problem can be extended to the case $1\le q<2$, it turns out that the main difficulty is to  prove an {\it intermediate} inequality similar to \eqref{bgn05}. Indeed, the absence of $H$-functions requires a completely new strategy.

The key idea is to introduce a new analytic-geometric functional involving both the energy of an obstacle problem and  geometric terms. Precisely, we add a perimeter term with {\it negative sign}, which may appear surprising for a minimization problem. However, this term is balanced by the obstacle energy. Indeed, the PDE and the geometric terms interact in the minimization process, 
which can be carried  out  in the framework of free discontinuity problems. We prove that the minimizer corresponds to a ball and, quite directly, this fact  provides  the intermediate quantitative inequality.

\section{Introduction of the new functional, main results and strategy of the proofs}
Let $q\in [1,2)$ and $\beta >0$ be given. Let $\Omega$ be a  bounded, open, Lipschitz set and $B$ a ball such that $|B|=|\Omega|$. Instead of working with the functionnal $\lambda_q(\Omega)$, we work with the following: for all $u\in H^1(\Omega), u\ge 0$, we define
\[E(u;\Omega)=\frac{1}{2}\int_{\Omega}|\nabla u|^2 d\Ln+\frac{\beta}{2}\int_{\partial\Omega} u^2d\Ls -\frac{1}{q}\int_{\Omega}u^q d\Ln, \]
and 
$$E(\Omega)= \min\{ E(u;\Omega) : u\in H^1(\Omega), u\ge 0\}.$$
$E(\Om)$ and $\lb_q(\Om)$ are linked by the relation:
\begin{equation}\label{bgn12.1}
E(\Omega)=\frac{q-2}{2q} \lb_q(\Omega)^\frac{q}{q-2},
\end{equation}

Here are the main results of the paper.
\begin{theorem}\label{mainresult}
For every, $\Om$, $\beta$ and $q$ as above, and let $u_\Om$ be a minimizer of $E(\cdot;\Om)$, then the intermediate inequality 
\begin{equation}\label{bgn08}
E(\Omega)-E(B) \ge  \frac{\beta}{2} \left(\inf_{x\in \Om} u_\Om(x)\right)^2 (\Per(\Om)-\Per(B))
\end{equation}
 holds true.
\end{theorem}

This intermediate inequality may have its own interest although it is not uniform.  Using the relation between $E$ and $\lambda_q$, we obtain the following inequality on $\lambda_q(\Omega)$

\[\frac{2-q}{2q}\left(\left(\frac{\lambda_q(\Omega)}{\lambda_q(B)}\right)^\frac{q}{2-q}-1\right)\lambda_q(\Omega) \ge  \frac{\beta}{2} \left(\inf_{x\in \Om} u(x)\right)^2 (\Per(\Om)-\Per(B)),
\]
where $u$ is a minimizer of the Rayleigh quotient \eqref{bgn07} that is normalized in $L^q$. Notice that when $q\rightarrow 2$, the left-hand side diverges while the right-hand side converges; we do not recover \eqref{bgn05} with this. However, we still obtain the following.

\begin{theorem}\label{mainresult.2}
For every, $\Om$, $\beta$ and $q$ as above:
\[ \lb_q(\Omega)-\lb_q(B)\geq C\mathcal{A}(\Omega)^2,\]
where $C>0$ depends on $n,\beta,q$ and $|\Omega|$, only.

\end{theorem}
In fact, both results are proved in a more general framework than stated above. They take the form of a  Sobolev-Poincar\'e inequality with trace terms in $SBV$ (see Theorem \ref{mainresult.1} in Section \ref{bgn09}) with improved constant. However, as most readers are interested only by the classical setting, we prefer to present our result for Lipschitz sets, and push technicalities in the second part of the paper.\bigbreak

To obtain the (intermediate) inequality in Theorem \ref{mainresult}, we will study a different problem that depends on a parameter $c\geq 0$. For every $c \ge 0$, for any nonnegative $u\in H^1(\Omega)$, we set
\begin{align*}
E^c(u;\Omega)&=E(c+u;\Omega)- \frac{\beta}{2} c^2 \Per(\Om) +\frac{c^q}{q} |\Om|\\
E^c(\Omega)&=\min\{ E^c(u;\Omega) : u\in H^1(\Omega), u\ge 0\}\\
&=\min\{E(u;\Omega):u\in H^1(\Omega),u\geq c\}- \frac{\beta}{2} c^2 \Per(\Om) +\frac{c^q}{q} |\Om|.
\end{align*}
This functional involves  both an obstacle problem and geometric terms. When we minimize  $E^c(\Omega)$ among sets of constant measure, the perimeter term, coming with negative sign, will interact with the solution of the obstacle problem, while the measure part does not play any role. At fixed $\Om$, the geometric terms do not play any role in the obstacle problem.

Clearly, for every $c \ge 0$
  \begin{equation}\label{bgn13}
E^c(\Om)\ge E(\Omega)-\frac{\beta}{2}c^2\Per(\Omega)+\frac{c^q}{q}|\Omega|.
\end{equation}
If $0\le c\le \inf_\Om u_\Om$ then the equality sign occurs in \eqref{bgn13} since $u_\Om-c$ is also solution of the obstacle problem.
If $c > \inf_\Om u_\Om$ then the solution of obstacle problem is different from $u_\Om-c$ and the inequality is strict. As an example, with $q=1$ and $\Omega=B_R$ (the ball of radius $R$), the minimizer $u$ of $E^c(\cdot;B_R)$ takes the form:
\[
u(x)=\left(\frac{R}{n\beta}-c\right)_+ +\frac{R^2-|x|^2}{2n},
\]
and
\[
E^c(B_R)=-\frac{|B_R|}{2}\left(\left(\frac{R}{n\beta}-c\right)_+ +\frac{R^2}{n(n+2)}\right).
\]

\medskip

The strategy to prove Theorems \ref{mainresult} and \ref{mainresult.2}  is based on the following steps.

\medskip
\noindent{\bf Step 1. Minimization of $\Om\mapsto E^c(\Om)$.}
 We prove that the ball minimizes $E^c$, i.e. 
  \begin{equation}\label{bgn10.1}
E^c(\Om) \ge E^c(B)\qquad \text{where } |\Om|=|B|.
\end{equation}

Before describing how we do it, we point out that \eqref{bgn10.1} leads quite directly to the intermediate quantitative inequality \eqref{bgn08} in Theorem \ref{mainresult}.
 Indeed, it is enough take $c=\inf_\Om u_\Omega$ in \eqref{bgn10.1}  and use \eqref{bgn13} to get
  \begin{equation}\label{eq1}
E(\Omega)-E(B)\geq \frac{\beta}{2}\left(\inf u_\Omega\right)^2 \left(\Per(\Omega)-\Per(B)\right).
\end{equation}

The proof of \eqref{bgn10.1} requires the most of our work, and is based the following arguments.
\begin{itemize}
\item We naturally relax the original shape optimization problem
$$\min\{E^c(\Om) : \Om \sq \RR^n, |\Om|=m\},$$
 as a new free discontinuity problem in the space of special functions of bounded variation (see \cite{Buc10, Buc15} and Section \ref{bgn04} below). Precisely, we consider
$$\min \{\Ec (v) : v \in SBV^{1/2}(\RR^n), |\{v >0\}|=m\},$$
where

\[\Ec(v)=\frac{1}{2}\int_{\Rn}|\nabla v|^2 d\Ln+\frac{\beta }{2}\int_{J_v}\left[(\underline{v}^2+2c\underline{v}) +(\overline{v}^2+2c\overline{v})\right]d\Ls -\int_{\Rn}\left(\frac{(c+v)^q-c^q}{q}\right) d\Ln \]
is chosen such that, if $v$ is in $H^1(\Omega)$ with $v\geq 0$ on $\Omega$ then, when  extended  by $0$ outside $\Omega$, we have
\[
\Ec(v)=E^c(v;\Omega)=E(c+v;\Omega)- \frac{\beta}{2} c^2 \Per(\Om) +\frac{c^q}{q} |\Om|.
\]
In particular, for $c=0$ we write $\mathcal{E}=\mathcal{E}^0$ and for any such function $u$ we have $\mathcal{E}^{0}(u)=E(u,\Omega)$.

\item The presence of the perimeter term with negative sign leads to a critical behavior of the boundary energy of the solution near the contact with the obstacle $c$. This is managed by approximation of the boundary energy: the terms of the form
$$\int_{J_v} (v^2+2cv) d\Ls \qquad\mbox{ are replaced by}\qquad \int_{J_v} (v^2+2cv^{1+\vps})d\Ls$$
 for $\vps >0$, small (see Definition \ref{bgn15}).
\item 
We prove that the minimizer of the approximating functional is a radial function with support on a ball.
First, we study qualitative properties of minimizers (non-degeneracy, closedness of the jump set,  radial symmetry) and, second, we show the existence of a solution. The approximation of the jump terms involving the parameter $\epsilon$ is in particular fundamental for the nondegeneracy result (see Lemma \ref{lemmaCaff}).
\item Pass to the limit $\vps \rightarrow 0$ and get that the minimizer of $\Ec$ is a radial function with support on a ball.
\end{itemize}
\medskip
\noindent{\bf Step 2.  Use of the selection principle to control uniformly $ \inf_{x\in \Om}u_\Om$.}
The intermediate inequality \eqref{eq1} together with the quantitative isoperimetric inequality leads to 
\[E(\Omega)-E(B)\geq \frac{\beta}{2}C_n\left(\inf_\Om u_\Omega\right)^2\mathcal{A}(\Omega)^2.\]
This quantitative inequality  is not uniform in $\Om$ since the right hand side is multiplied by $\inf_\Om u_\Omega$.
We  regularize $\Omega$ by replacing it with $\Omega^\text{opt}$, a minimizer of 
\[\omega\mapsto E(\omega)+k|\omega|\]
among all $\omega\subset \Omega$ for some small enough $k>0$. Following the main lines of \cite{BFNT18}, we prove that $\displaystyle \inf_{x \in {\Omega^\text{opt}}} u_{\Omega^\text{opt}}(x)\ge{\alpha}>0$ where $\alpha$   depends on $n,\beta,q$ and $|\Om|$, while $\mathcal{A}(\Om^{\text{opt}})$ is comparable to $\mathcal{A}(\Om)$.  This will conclude the proof. 

\medskip
The paper is organized as follows. In Sections \ref{bgn04} we study the minimization of the geometric functional $\Om \mapsto E^c(\Om)$, by relaxation in $SBV$ and approximation. In Section \ref{bgn12} we prove that the minimizer corresponds to a ball. These two sections concentrate most of the technicalities of the paper. In the last section  we prove Theorems \ref{mainresult} and \ref{mainresult.2}.

 It should be noted that for the quantitative inequality  with Dirichlet boundary condition,  in \cite{bra15} the authors reduced their study to the sole study of the torsion functional $\Om \mapsto \inf_{u\in H^1_0(\Omega)}\int_{\Omega}\left(\frac{1}{2}|\nabla u|^2-u\right)d\Ln$ (corresponding to the case $q=1$). This is due to a hierarchy of the eigenvalues relying on the Kohler-Jobin inequality (see \cite[Chapter 7, Section 7.8.1]{BDP17}). 
 To our knowledge, there is no such inequality with Robin boundary conditions, the reason for which we have to directly work on the general case.

\section{Analysis of the analytic-geometric functional}\label{bgn04}
In this section we study the minimization of $\Om \mapsto E^c(\Om)$ in the class of open, bounded, Lipschitz sets of measure $m$. For that purpose, we introduce the relaxed form of the functional in 
the space of special functions of bounded variation. We refer the reader to \cite{AFP} for an introduction to the SBV space, as subspace in $BV(\RR^n)$. Below, we denote by $Du$  the distributional gradient of $u$ and recall that
$$SBV(\RR^n) =  \{u \in BV(\RR^n) : Du \mbox{ is absolutely continuous with respect to } dx + \Ls\lfloor_{J_u}\}.$$
The following space was introduced in \cite{Buc10},
$$SBV^{1/2}(\mathbb{R}^n)=\{u\in BV_{loc} (\RR^n) : u \ge 0,  u^2 \in SBV(\RR^n)\}. $$
We refer to \cite{Buc10,Buc15} for the main properties of $SBV^{1/2}(\mathbb{R}^n)$.  In particular, we recall the following Poincar\'e inequality with trace term proved in  \cite{Buc15} for $q\in [1,2]$; for all $u \in SBV^{1/2}(\mathbb{R}^n)$, $|\{u>0\}|\le m$:
\begin{equation}\label{bgn16}
 \lb_q(B^m) \left(\int_{\RR^n} u^q dx\right)^\frac{2}{q} \le \int_{\Rn}|\nabla u|^2 d\Ln+\beta \int_{J_u}(\underline{u}^2 +\overline{u}^2)d\Ls.
\end{equation}
Where $B^m$ is the ball of volume $m$, and for every $x\in J_u$, $\underline{u}(x)$ and $\overline{u}(x)$ refer to the lower and upper approximate limits of $u$ at $x$. The constant $ \lb_{ q}(B^{ m})$ is optimal.

\begin{definition}\label{bgn15} Let $c \ge 0$. For any function $v\in SBV^{1/2}(\mathbb{R}^n)$ we set:

\[\Theta(v)=\left(\frac{(c+v)^q-c^q}{q}\right).\]

We introduce the following regularization of $\Ec$
\[
\Ece(v):=\frac{1}{2}\int_{\Rn}|\nabla v|^2 d\Ln+\frac{\beta }{2}\int_{J_v}\left[(\underline{v}^2+2c\underline{v}^{1+\epsilon}) +(\overline{v}^2+2c\overline{v}^{1+\epsilon})\right]d\Ls -\int_{\Rn}\Theta(v) d\Ln, 
\]
which is well posed thanks to \eqref{bgn16}. We shall also work with the  penalised version
\[\Ecek(v)=\Ece(v)+k|\{ v>0\}|,\]
where $k>0$ is a positive constant.
\end{definition}

Below is the key result of this section. Let us denote by
$$
{\mathcal U}_{m}:=\{u\in SBV^{1/2}(\mathbb{R}^n) : |\{u>0 \}| =m\}
$$
the class of admissible functions. 

\begin{theorem}\label{bgn03}
For every $\vps >0$ and $c\ge 0$ the solution of 
$$\min \{ \Ece (u) : u \in {\mathcal U}_{m}\}$$
is a radial function and its support is a ball of measure $m$.
\end{theorem}

The rest of the section is devoted to the proof of this theorem.
The strategy is as follows.
\begin{itemize}[label=\textbullet]
\item We first assume the existence of a minimizer $u$ and  study its properties to arrive at the conclusion that it is a radial function with the support being a ball of measure $m$.
\item We then prove the existence of a minimizer by analyzing a minimizing sequence using the a priori properties proved before. The key point is to show that, up to a subsequence and up to translations, a minimizing sequence necessarily has to concentrate the mass around the origin and to converge.\end{itemize}


\subsection{Preparatory results}
We write $\Ece$ as
\[\Ece(u)=Q(u)+N^{c,\epsilon}(u)-\int_{\mathbb{R}^n}\Theta(u)d\Ln,\]
where
\[Q(v)=\frac{1}{2}\int_{\Rn}|\nabla v|^2 d\Ln+\frac{\beta }{2}\int_{J_v}(\underline{v}^2+\overline{v}^2)d\Ls\]
and
\[N^{c,\epsilon}(v)=\beta c \int_{J_v}(\underline{v}^{1+\epsilon}+\overline{v}^{1+\epsilon})d\Ls.\]

For any non trivial $u\in SBV^{1/2}$ such that $\Ece(u)<\infty$, we have for small $t>0$
\[\Ece(tu)<0.\]
 While $\Ece(u)$ is not necessarily positive, its terms coming with positive and negative sign, control each other in certain cases. We summarize this observation as follows.
\begin{lemma}\label{ResultControl}
For every function  $u\in {\mathcal U}_{m}$ there is constant $C>0$, depending on $n,m$, $\beta,q,c,\Ece(u)$ such that
\[Q(u)+N^{c,\epsilon}(u)+\Vert \Theta (u)\Vert_{L^1}<C.\]
\end{lemma}

\begin{proof}
For any such $u$, we have, using \eqref{bgn16}
\[\int_{\mathbb{R}^n}ud\Ln\leq \lambda_1(B^m)^{-\frac{1}{2}}\left(\int_{\mathbb{R}^n}|\nabla u|^2d\Ln+\beta\int_{J_u}(\underline{u}^2+\overline{u}^2)d\Ls\right)^{\frac{1}{2}}\]
and
\[\left(\int_{\mathbb{R}^n}u^q d\Ln\right)^{1/q}\leq \lambda_q(B^m)^{-\frac{1}{2}}\left(\int_{\mathbb{R}^n}|\nabla u|^2d\Ln+\beta\int_{J_u}(\underline{u}^2+\overline{u}^2)d\Ls\right)^{\frac{1}{2}},\]
where $B^m$ is the ball of volume $m$. We then know that for a certain constant $C$, 
\[\int_{\mathbb{R}^n}\Theta(u)d\Ln\leq C (Q(u)^{1/2}+Q(u)^{q/2}),\]
so that
\[Q(u)+N^{c,\epsilon}(u)-C (Q(u)^{1/2}+Q(u)^{q/2})\leq \Ece(u).\]
This means that $Q(u)+N^{c,\epsilon}(u)$ is bounded by a constant that depends only on $\Ece(u)$ and $C$ (which depends on $n,m,\beta,q,c$); this proves the result.
\end{proof}

Let us now study the monotonicity of $\Ece$ on scale change. This will imply that minimizing $\Ece(u)$ under the constraint $|\{ u>0\}|=m$ or $|\{ u>0\}|\leq m$ is equivalent.
\begin{lemma}\label{bgn17}
Let $u\in SBV^{1/2}(\RR^n)$ be a non-trivial function and $t>1$. Then:
\[\Ece\left(u\left(\cdot/t\right)\right)<t^{n}\Ece(u)\]
\end{lemma}
\begin{proof}
It is immediate by a change of variable and using $t^{n-1},t^{n-2}<t^n$.
\end{proof}

We have to compare a minimizer $u$ with a function whose support does not necessarily have the same measure. The following lemma allows us to use $\Ecek$, the volume-penalised version of $\Ece$ with a suitable $k$.

\begin{lemma}\label{ResultMeasurePenalization}
Let $u\in {\mathcal U}_{m}$ be a minimizer of the functional $\Ece$. Then
\begin{itemize}[label=\textbullet]
\item For $k=\frac{-\Ece(u)}{m}$, $u$ is a minimizer of $\Ecek$ in the class
\[\{ v\in SBV^{1/2}(\RR^n):|\{v>0\}|\leq m\}.\]
\item For $k=\frac{2\Vert\Theta(u)\Vert_{L^1}}{m}$, $u$ is a minimizer of $\Ecek$ in the class
\[\{ v\in SBV^{1/2}(\RR^n):|\{v>0\}|\geq m,\ \Vert \Theta(v)\Vert_{L^1}\leq 2\Vert \Theta(u)\Vert_{L^1}\}.\]
\end{itemize}
\end{lemma}

\begin{proof}
\textbf{First case}: Let $v$ be such a function, we write $\frac{|\{ v>0\}|}{m}=1-\eta$ with $\eta\in ]0,1[$.
Let 
$$
w(x):=v((1-\eta)^{1/n}x).
$$ 
 Notice that $w$ is in $\Um$ with 
\begin{align*}
\Ece(w)&=(1-\eta)^{-1+\frac{2}{n}}\frac{1}{2}\int_{\Rn}|\nabla v|^2 d\Ln
\\ &+(1-\eta)^{-1+\frac{1}{n}}\frac{\beta }{2}\int_{J_v}\left[(\underline{v}^2+2c\underline{v}^{1+\epsilon}) +(\overline{v}^2+2c\overline{v}^{1+\epsilon})\right]d\Ls-(1-\eta)^{-1}\int_{\mathbb{R}^n}\Theta(v) d\Ln,
\end{align*}
so that $\Ece(w)\leq (1-\eta)^{-1}\Ece(v)$.  The minimality of $u$ yields  
\[(1-\eta)\Ece(u)\leq \Ece(v),\]
which, with our definition of $k$ and $\eta$, is exactly
\[\Ecek(u)\leq \Ecek(v).\]

\noindent \textbf{Second case}: We proceed in the same way. Let us consider such a function $v$, write $\frac{|\{ v>0\}|}{m}=1+\eta$ where $\eta>0$. Let
$$
w(x):=v((1+\eta)^{1/n}x).
$$ 
 Again  $w$ is in $\Um$ with
\begin{align*}
\Ece(w)&=(1+\eta)^{-1+\frac{2}{n}}\frac{1}{2}\int_{\Rn}|\nabla v|^2 d\Ln\\
&+(1+\eta)^{-1+\frac{1}{n}}\frac{\beta }{2}\int_{J_v}\left[(\underline{v}^2+2c\underline{v}^{1+\epsilon}) +(\overline{v}^2+2c\overline{v}^{1+\epsilon})\right]d\Ls-(1+\eta)^{-1}\int_{\mathbb{R}^n}\Theta(v)d\Ln.
\end{align*}
For any $\eta>0$, we have the inequalities
\[(1+\eta)^{-1+\frac{2}{n}},(1+\eta)^{-1+\frac{1}{n}}\leq 1,\ \ (1+\eta)^{-1}\geq 1-\eta.\]

From the minimality of $u$ and the choice of $v$ we can  write
\begin{align*}
\Ece(u)&\leq \Ece(w)&\\
& \leq \Ece(v)+\eta \int_{\mathbb{R}^n}\Theta(v) d\Ln\\
&\leq \Ece(v)+2\eta\int_{\mathbb{R}^n}\Theta(u) d\Ln,&
\end{align*}
which, following the definition of $k$ and $\eta$ yields
\[\Ecek(u)\leq \Ecek(v).\]
\end{proof}
\bigbreak

\subsection{Nondegeneracy of the minimizers} In this part, we prove that a minimizer lies above a strictly positive threshold, on the set where it is non vanishing. We refer to \cite{Buc14}, \cite{Caf16}, \cite{Buc15}, \cite{BFNT18} for similar arguments.

\begin{lemma}\label{lemmaCaff}
Let $u\in {\mathcal U}_{m}$ be a minimizer of the functional $\Ece$. There exists $\delta>0$, depending on $n,c,\beta,m,\epsilon,\Ece(u)$, such that $u\geq\delta 1_{\{u>0\}}$.
\end{lemma}

\begin{proof}
In this proof, which follows the main lines of \cite[Theorem 3.2]{Caf16}, we denote by $C$ a positive constant that may change from line to line which  depends on the parameters only. Let $k$ be defined as in the first part of the previous result.
We introduce the following function
\[
f_{t_{\text{min}}}(t)=\int_{\{ t_{\text{min}} \leq u\leq t\}}u^{\epsilon}|\nabla u|d\Ln,
\]
where $0<t_{\text{min}}<t$. Since $(u-t_{\text{min}})_+$ is in $SBV$, we can apply the coarea area for $SBV$ functions (see \cite{Caf16}) to get
\[
f_{t_{\text{min}}}(t)=\int_{t_{\text{min}}}^t s^{\epsilon}\text{Per}(\{ u\leq s\}; \mathbb{R}^n\setminus J_u)ds.
\]
Define $f(t)=f_{t_{\text{min}}=0}(t)$; by taking the limit $t_{\text{min}}\rightarrow 0^+$ in the above formula, we see that by monotone convergence
\[
f(t)=\int_{\{ 0 \leq u\leq t\} }u^{\epsilon}|\nabla u|d\Ln=\int_{0}^t s^{\epsilon}\text{Per}(\{ u\leq s\}; \mathbb{R}^n\setminus J_u)ds.
\]

The first form yields the fact that $f$ is bounded (for bounded $t$ at least). Indeed, it is controlled by the positive part of $\Ece(u)$, and more precisely by $Q(u)+k|\{ u>0\}|$
\[f(t)\leq t^{\epsilon}\int_{\{ u\leq t\} }|\nabla u|d\Ln\leq t^{\epsilon}|\{ 0<u\leq t\}|^{1/2}\left(\int_{\{ 0<u\leq t\}}|\nabla u|^2d\Ln\right)^{1/2}.
\]
We can then apply Lemma \ref{ResultControl} to infer that the positive part of $\Ece(u)$ is controlled by a constant that depends on $\Ece(u)$ and the parameters of the problem. 
As $u$ is a minimizer, $f$ is bounded by a constant that only depends on the parameters.
\par
Next, we use the optimality of $u$ against $u1_{\{u>t\}}$ in view of Lemma \ref{ResultMeasurePenalization}, which gives
\[\Ecek(u)\leq \Ecek\left(u1_{\{u>t\}}\right)\]
or, after computations,
\begin{multline*}
Q\left(u1_{\{u\leq t\}}\right)+N^{c,\epsilon}\left(u1_{\{u\leq t\}}\right)+k|\{ 0<u\leq t\}|-\int_{\{ 0<u\leq t\}}\Theta(v)d\Ln\\
\leq \frac{\beta}{2}\int_{\partial^*\{ u>t\} \setminus J_u}(u^2+2cu^{1+\epsilon})d\Ls.
\end{multline*}

In the rest of the proof, we will only consider $t$ small enough
\[t<\Theta^{-1}(k/2)\wedge c^{\frac{1}{1-\epsilon}},\]
which allows us to write 
$$
k|\{ 0<u\leq t\}|-\int_{\{ 0<u\leq t\}}\Theta(v)d\Ln\geq \frac{k}{2}|\{ 0<u\leq t\}|,
$$
and
\begin{align*}
\frac{\beta}{2}\int_{\partial^*\{ u>t\} \setminus J_u}(u^2+2cu^{1+\epsilon})d\Ls&=\frac{\beta}{2}(t^2+2ct^{1+\epsilon})\Ls(\partial^*\lbrace u>t\rbrace \setminus J_u)\\
&=\frac{\beta}{2}(t^{2-\epsilon}+2ct)f'(t)\leq \frac{3}{2}\beta c t f'(t).
\end{align*}
With these, the optimality condition becomes
\begin{equation}\label{ConditionOptu}
Q\left(u1_{\{u\leq t\}}\right)+N^{c,\epsilon}\left(u1_{\{u\leq t\}}\right)+\frac{k}{2}|\{ 0<u\leq t\}|\leq\frac{3}{2}\beta c t f'(t).
\end{equation}
By H\"older,
\begin{multline*}
f(t)\leq |\{ 0<u\leq t\}|^{\frac{n-(n-1)\epsilon}{2n}}\Vert 1_{\{0<u<t\}}\nabla u\Vert_{L^2}\ \Vert 1_{\{0<u\leq t\}}u^2\Vert_{L^{\frac{n}{n-1}}}^\frac{\epsilon}{2}\\
\leq C\Big[ t f'(t)\Big]^{1-\frac{n-1}{2n}\epsilon}\ \Vert u^21_{\{0<u\leq t\}}\Vert_{L^{\frac{n}{n-1}}}^\frac{\epsilon}{2}.
\end{multline*}
To estimate the last factor, we use the continuity of the embedding $BV(\RR^n)\hookrightarrow L^{\frac{n}{n-1}}(\RR^n)$  and \eqref{ConditionOptu} to get
\begin{multline*}
c_n\Vert u^2 1_{\{0<u\leq t\}}\Vert_{L^{\frac{n}{n-1}}}\leq [ u^2 1_{\{0<u\leq t\}}]_{BV}
=|D(u^21_{\{0<u\leq t\}})|(\mathbb{R}^n)\\
=\int_{\{ 0<u\leq t\}}2u|\nabla u|\ d\Ln+\int_{J_u\cap\{ u < t\}^0}(\underline{u}^2+\overline{u}^2)d\Ls+\int_{J_u\cap\partial^*\{ u>t\}}\underline{u}^2d\Ls+\int_{\partial^*\{ u>t\}\setminus J_u}u^2d\Ls\\
\leq C\Big[Q\left(u1_{\{u\leq t\} }\right)+|\{0<u\leq t\}|+t^2\text{Per}(\{ u>t\};\mathbb{R}^n\setminus J_u)\Big]
\leq C t f'(t).
\end{multline*}
 Coming back to the estimate of $f(t)$, we obtain
\[f(t)\leq C\Big[t f'(t)\Big]^{1+\frac{\epsilon}{2n}}.\]
This implies that, for all $t$ such that $f(t)>0$, we have:
\[\frac{d}{dt}\Big[f(t)^{\frac{\epsilon}{2n+\epsilon}}\Big]\geq \frac{1}{Ct}.\]
Let $t_0$ be such that $f(t_0)>0$ and $t_1:=\min(c^{\frac{1}{1-\epsilon}},\Theta^{-1}(k/2))$. We integrate on $[t_0,t_1]$
\[(f(t_1)^{\frac{\epsilon}{2n+\epsilon}}\geq) f(t_1)^{\frac{\epsilon}{2n+\epsilon}}-f(t_0)^{\frac{\epsilon}{2n+\epsilon}}\geq \frac{1}{C}\log(t_1/t_0).\]
Using our uniform bound on $f(t_1)$, we obtain a lower bound $\delta$ on $t_0$ that only depends on the parameters of the problem, meaning that $f(\delta)=0$ for an explicit $\delta>0$. We apply the optimality condition in $t=\delta$:
\[
Q\left(u1_{\{u\leq \delta\}}\right)+N^{c,\epsilon}\left(u1_{\{u\leq \delta\}}\right)+\frac{k}{2}|\{ 0<u\leq \delta\}|\leq 0,\]
and so $|\{ 0<u\leq \delta\}|=0$.

\end{proof}

\subsection{Closedness of the jump set}
 We prove below that the support of the minimizer $u$ is an open set $\Omega$ with finite perimeter. For this we prove that the jump set of $u$ is closed as $J_u$ is identified with $\partial \Om$. 
\begin{lemma}\label{bounded}
Let $u\in {\mathcal U}_{m}$ be a minimizer of the functional $\Ece$. Then $u\in SBV(\RR^n)\cap L^\infty(\RR^n)$ and $\Ls(J_u)<\infty$. Moreover, $\Ls(\overline{J_u}\setminus J_u)=0$.
\end{lemma}

\begin{proof}
We divide the proof in several steps.
\vskip10pt\noindent{\bf Step 1.} We prove  that $u\in L^\infty(\RR^n)$. Let us set 
$$
u_M=(u-M)_+,\qquad f(M)=\int_{\mathbb{R}^n}u_Md\Ln\qquad\text{and}\qquad \alpha(M)=|\{ u_M>0\} |,
$$ 
 and  let $\lambda_{2,\beta}(B^m)$ be the best constant in \eqref{bgn16} with $q=2$ and $\beta$ as  boundary parameter, among sets with mass $m$. We suppose that $u$ is not bounded, and so that $f(M)>0$ and $\alpha(M)>0$ for all $M>0$. Let 
$$
g(M)=\frac{ \lambda_{2,\alpha(M)^{1/n}\beta} (B^1)}{\alpha(M)^{1/n}}.
$$ 
The results we will use here is that
\[
 \lambda_{2,\beta}(B^{\alpha(M)})=\alpha(M)^{-\frac{1}{n}}g(M)\text{ with }\liminf_{M\rightarrow \infty}g(M)>0.
\]
We test the optimality of the function $u$ against the function $u\wedge M$. We get
\begin{multline*}
\frac{1}{2}\int_{\{ u>M\}}|\nabla u|^2d\Ln
 +\frac{\beta}{2}\int_{\{ M\leq \underline{u}<\overline{u}\}}(\overline{u}^2+2c\overline{u}^{1+\epsilon}+\underline{u}^2+2c\underline{u}^{1+\epsilon})d\Ls\\
 +\frac{\beta}{2}\int_{\{ \underline{u}<M\leq\overline{u}\}}(\overline{u}^2+2c\overline{u}^{1+\epsilon}-M^2-2cM^{1+\epsilon})d\Ls
\leq \int_{\{ u>M\} }(\Theta(u)-\Theta(M))d\Ln.
\end{multline*}
This implies:
\begin{equation}\label{1eq}
Q(u_M)\leq \int_{\{ u>M\} }(\Theta(u)-\Theta(M))d\Ln
\end{equation}
and \eqref{bgn16} gives
\begin{equation}\label{2eq}
g(M)\alpha(M)^{-\frac{1}{n}}\int_{\mathbb{R}^n}u_M^2d\Ln\leq Q(u_M).
\end{equation}
Using $1\leq q<2$, we know that for all $1\leq a\leq b$, we have $b^q-a^q\leq b^2-a^2$, which implies, for $M\geq 1$, that
\begin{align*}
\int_{\{ u>M\} }(\Theta(u)-\Theta(M))d\Ln&\leq \frac{1}{q}\int_{\{ u>M\} }((c+M+u_M)^2-(c+M)^2)d\Ln\\
&= \frac{1}{q}\int_{\{ u>M\} }(u_M^2+2(c+M)u_M)d\Ln.
\end{align*}
Combining this with the estimates \eqref{1eq},\eqref{2eq}, we get:
\[\left(g(M)\alpha(M)^{-\frac{1}{n}}-\frac{1}{q}\right)\int_{\mathbb{R}^n}u_M^2d\Ln\leq \frac{2}{q}(c+M)f(M).\]
Since $\liminf_{M\rightarrow\infty}g(M)>0$ and $\alpha(M)\underset{M\rightarrow \infty}{\longrightarrow} 0$, we know that for all big enough $M$, 
\[g(M)\alpha(M)^{-\frac{1}{n}}-\frac{1}{q}\geq \frac{1}{2}g(M)\alpha(M)^{-\frac{1}{n}}.\]
Holder's inequality gives:
\[\alpha(M)^{-1}f(M)^2\leq\int_{\mathbb{R}^n}u_M^2d\Ln.\]
Thus we get:
\[\frac{1}{2}g(M)\alpha(M)^{-\frac{n+1}{n}}f(M)^2\leq \frac{2}{q}(c+M)f(M), \]
which can be rewritten as
\[\left(\frac{q}{4}\frac{g(M)}{c+M}\right)^{\frac{n}{n+1}}\leq \alpha(M)f(M)^{-\frac{n}{n+1}}.\]
The left side is not integrable because $\liminf g>0$. Since $f(M)\rightarrow 0$ and $f'(M)=-\alpha(M)$, the right side is integrable (its integral on $[M_0,+\infty[$ is $\frac{1}{n+1}f(M_0)^{\frac{1}{n+1}}<\infty$): this is a contradiction. We deduce that $u$ is bounded by a certain constant $M>0$.

\vskip10pt
\noindent{\bf Step 2.} We get $\Ls(J_u)<\infty$. Indeed, by Lemma \ref{lemmaCaff} we have  $u>\delta 1_{\{u>0\}}$ which implies
\[\delta^2 \Ls(J_u)\leq \int_{J_u}(\overline{u}^2+\underline{u}^2)d\Ls<\infty.\]
\medskip

\vskip10pt
\noindent{\bf Step 3.} The function $u$ belongs to $ SBV(\RR^n)$. The proof follows the same arguments as in \cite{Buc15}, by considering the function $u^\eta=\sqrt{u^2+\eta^2}$ for $\eta\rightarrow 0$.

\vskip10pt
\noindent{\bf Step 4.} We show the closedness of the jump set $\Ls(\overline{J_u}\setminus J_u)=0$. Following \cite{Buc14}, we only need to show that $u$ is a local almost quasi-minimizer of the Mumford-Shah functional. We use $\delta 1_{\{u>0\}}\leq u\leq M$ to prove this.

Indeed, consider $v\in SBV(\RR^n)$  such that $\{ u\neq v\}\Subset B_r$ for a ball $B_r$ of radius $r>0$ small enough. Let $v'=v\wedge M$. Since $u\leq M$, then $\{ u\neq v'\}\Subset B_r$ and $v'$ still belongs to $SBV$.
Applying Lemma \ref{ResultMeasurePenalization}, either $|\{ v'>0\}|\leq m$ and we are in the first case  or $|\{ v'>0\}|\geq m$ and we are in the second one. In the second case we need to verify that $\Vert \Theta(v')\Vert_{L^1}\leq 2 \Vert \Theta(u)\Vert_{L^1}$ ; this is true for small enough $r$ since $\Vert \Theta(v')\Vert_{L^1}\leq \Vert \Theta(u)\Vert_{L^1}+|B_r|\Theta(M)$. Thus for a small enough $r$ there exists $k>0$ depending only on the parameters such that 
\[\Ecek(u)\leq\Ecek(v').\]
This can be rewritten
\begin{align*}
\frac{1}{2}&\int_{B_r}|\nabla u|^2d\Ln-\int_{B_r}\Theta(u)d\Ln\\
&+\frac{\beta}{2}\int_{J_u\cap B_r}\left[(\underline{u}^2+2c\underline{u}^{1+\epsilon}) +(\overline{u}^2+2c\overline{u}^{1+\epsilon})\right]\,d\Ls+k|\{ u>0\}\cap B_r|\\
\leq\frac{1}{2}&\int_{B_r}|\nabla v'|^2d\Ln-\int_{B_r}\Theta(v')d\Ln\\
&+\frac{\beta}{2}\int_{J_{v'}\cap B_r}\left[(\underline{v'}^2+2c\underline{v'}^{1+\epsilon}) +(\overline{v'}^2+2c\overline{v'}^{1+\epsilon})\right]\,d\Ls+k|\{ v'>0\}\cap B_r|.
\end{align*}
Using $u>\delta 1_{\{u>0\}}$ and $v'\leq M$, as well as $|\nabla v'|\leq |\nabla v|$, $J_{v'}\subset J_v$, we get
\begin{multline*}
\int_{B_r}|\nabla u|^2d\Ln+\beta(\delta^2+2c\delta^{1+\epsilon})\Ls(J_u\cap B_r)\\
\leq \int_{B_r}|\nabla v|^2d\Ln+2\beta (M^2+2cM^{1+\epsilon})\Ls(J_v\cap B_r)+2\alpha_n \left(k+\Theta(M)\right) r^n,
\end{multline*}
where $\alpha_n=|B_1|$. Up to a renormalization of $u$, this is exactly the definition of a local quasi-almost minimizer. Following \cite[Theorem 3.1]{Buc14},  this implies
\[\Ls(\overline{J_u}\setminus J_u)=0.\]
\end{proof}

\begin{lemma}
\label{lem:support}
Let $u\in {\mathcal U}_{m}$ be a minimizer of the functional $\Ece$, then there exists an open domain $\Omega$ such that the following items hold true.
\begin{itemize}[label=\textbullet]
\item[(a)] $\partial\Omega=\overline{J_u}$, $\Ls(\partial\Omega\setminus J_u)=0$ and $u=0$ a.e. on $\mathbb{R}^n\setminus\Omega$
\item[(b)] $u_{|\Omega} \in H^1(\Omega)$ and verifies $\delta<u_{|\Omega}<M$ for certain constants $\delta,M>0$ and 
\[-\Delta u=(c+u)^{q-1}\text{ in }{\mathcal D}'(\Omega).\]

\end{itemize}
In particular, $u$ is analytic on its support.
\end{lemma}
\begin{proof}
The proof is the same as Theorem 6.15 of \cite{Buc15}.
\end{proof}
\medskip

\subsection{The optimal function is radially symmetric} Now we prove that the optimal function is radial and is supported on a ball of measure $m$. 
\begin{lemma}
Let $u\in {\mathcal U}_{m}$ be a minimizer of the functional $\Ece$, then $u$ is radial and its support is a ball.
\end{lemma}
\begin{proof}
Using Lemma \ref{lem:support}, we know that $u$ is an analytic function on an open domain $\Omega$ with finite perimeter. We divide the proof in several steps.

\medskip
\noindent {\bf Step 1.} {\bf The set $\Om$ is connected.}
We first show that $\Omega$ is connected: suppose that $\Omega=V\sqcup W$ for two open sets $V,W$, then we write $v=u1_V$ and $w=u 1_W$: these functions are in $\mathcal{U}_{|V|}$ and $\mathcal{U}_{|W|}$. We also set
$$
\widetilde{v}(x):=v\left(\left[\frac{|V|}{m}\right]^{1/n}x\right)\qquad\text{and}\qquad
\widetilde{w}(x):=w\left(\left[\frac{|W|}{m}\right]^{1/n}x\right).
$$
The functions $\widetilde{v}$ and $\widetilde{w}$ are in $\Um$ so, by comparison to $u$ and Lemma \ref{bgn17},

\[\Ece(u)\leq \frac{|V|}{m}\Ece(\widetilde{v})+\frac{|W|}{m}\Ece(\widetilde{w})< \Ece(v)+\Ece(w)=\Ece(u).\]
 This is a contradiction which implies that $\Omega$ is connected.\bigbreak

\medskip
\noindent {\bf Step 2.}  {\bf The function $\mathbf{u}$ is locally radial.} Here we call {\it locally radial} any function $u$ that is locally the restriction of a radial function on a support that is not necessarily radial; when $u$ is smooth on its support, which is the case here, it is equivalent, up to a translation, to $\nabla u(x)$ being proportional to $x$ for every $x\in\Omega$.\bigbreak

We first build a symmetrized version of $\Omega$ called $\Omega ^s$ such that any hyperplane going through the origin cuts $\Omega^s$ in two parts of same volume. The procedure is the following: take $\lambda_1\in\mathbb{R}$ such that the hyperplane $\{ x_1=\lambda_1\}$ cuts the support of $u$ in two parts of equal volume. Let $S_1$ be the symmetry across this hyperplane, and
\[u_1^+=\begin{cases}u & \text{ in }\{ x_1> \lambda_1\} \\ u\circ S_1 & \text{ in }\{ x_1 <\lambda_1\}\end{cases}\]
\[u_1^-=\begin{cases}u & \text{ in }\{ x_1< \lambda_1\} \\ u\circ S_1 & \text{ in }\{ x_1 >\lambda_1\}\end{cases}\]
Then $u_1^\pm$ both have a support of volume $m$ and $\Ece(u_1^+)+\Ece(u_1^-)\leq 2\Ece(u)$ (the inequality is only there because some part of $J_u$ lying on $\{ x_1=\lambda_1\}$ might be deleted in $u_1^\pm$: when this does not happen, there is equality). In particular $u_1^+$ and $u_1^-$ are also minimizers of $\Ece$ in $\Um$. Now apply the same procedure to $u_1^+$ across an hyperplane $\{ x_2=\lambda_2\}$, and so on: we get in the end a minimizer $u^s\in \Um$ that coincides with $u$ on the quadrant $\{ x:x_i\geq\lambda_i,\ \forall i=1,\hdots,n\}$, and that is symmetric relative to every hyperplane $\{ x_i=\lambda_i\}$.\bigbreak

Without loss of generality, we can suppose that every $\lambda_i$ is $0$. Let $\Omega^s$ be the support of $u^s$, we know that it is symmetric relative to every $\{ x_i=0\}$. The composition of all those symmetries is the central symmetry relative to the origin, and so $\Omega^s$ has a central symmetry. In particular, any hyperplane that goes through the origin cuts the volume of $\Omega^s$ in half.
\bigbreak

We now show that $u^s$ is locally radial: this implies that $u$ is locally radial because $u=u^s$ on $\Omega\cap\Omega^s(\neq\emptyset)$, $u$ and $u^s$ are analytic, and $\Omega$ is connected, so we can use analytic continuation. To show that $u^s$ is locally radial, take $\Pi$ any hyperplane going through $0$, and $S$ the symmetry across this hyperplane. Let $\Pi^+$ and $\Pi^-$ be the two half-space defined from $\Pi$, we define

\[u^{s,\pm}=\begin{cases}u^s & \text{ in }\Pi^\pm, \\ u\circ S & \text{ in }\Pi^{\mp}.\end{cases}\]

Moreover, $\Ece(u^{s,+})+\Ece(u^{s,-})\leq 2\Ece(u^s)$, so $u^{s,+}$ (and $u^{s,-}$) is a minimizer. Since it is a minimizer, it is $\mathcal{C}^1$ on its support which is an open set. $u^s$ and $u^{s,+}$ are both $\mathcal{C}^1$ and coincide on $\Pi^-\cap\Omega$: this implies that $\nabla u^s$ is purely normal to $\Pi$ on $\Pi\cap\Omega$. Since this is true for any hyperplane, we have shown that $u^s$ is locally radial, which implies that $u$ is locally radial by analyticity.\bigbreak

We know that $u$ is locally radial on a connected open set, and verifies $-\Delta u=(c+u)^{q-1}$, which means that it can be written $u(x)=\psi(|x|)1_\Omega(x)$ where $\psi$ verifies the following ordinary differential equation: 
\begin{equation}
\label{eq:ODE}
\psi''+\frac{n-1}{r}\psi'+(\psi+c)^{q-1}=0.
\end{equation}

\medskip
\noindent {\bf Step 3.}
{\bf The set $\Omega$ is radial}. To prove this we use the argument developed in \cite{Buc19}. We show that the topological boundary of $\Omega$ coincides with the measure theoretical one, up to a $\Ls$-negligible set. Since $u=\psi(|\cdot|)1_{\Omega}$ where $\psi$ is analytic and bounded below, the singular part (relative to the Lebesgue measure) of the derivative of $u$ is:
\[D^s u=\psi(|\cdot|)D^s1_{\Omega}.\]
And so its support (which is $\overline{J_u}$ up to a $\Ls$-negligible set) is the support of $D^s1_\Omega$, which is $\partial^*\Omega$ up to a $\Ls$-negligible set. Since $\partial\Omega=\overline{J_u}$, it means that $\Ls(\partial\Omega\setminus\partial^*\Omega)=0$, and so $\Omega$ has no inner boundaries. This will allow us to use symmetrization technique on $\Omega$ in the framework of sets with finite perimeter, to show that $\Omega$ itself is radial.

 Let $\widetilde{\Omega}$ be the symmetrization of $\Omega$ by spherical caps in the direction $e_1$, and $\widetilde{u}(x)=\psi(|x|)1_{\widetilde{\Omega}}(x)$. We aim to show that $\Omega$ is necessarily radial, which is implied by the property $|\widetilde{\Omega}\cap\{ x_1<0\}|=|\widetilde{\Omega}\cap\{ x_1>0\}|$. We suppose it is not the case, then there is a point $p=\lambda e_1$ for $\lambda>0$ such that all hyperplanes going through $p$ cuts $\widetilde{\Omega}$ in two parts with same volume. This is because it is the case for $n$ orthogonal hyperplanes, namely the $\{ x_i=0\}$ for $i\geq 2$ (by property of the spherical cap rearrangement), and a certain $\{ x_1=\lambda\}$ for $\lambda>0$ (because $|\widetilde{\Omega}\cap\{ x_1<0\}|<|\widetilde{\Omega}\cap\{ x_1>0\}|$).\bigbreak

Moreover, $\widetilde{u}$ is a minimizer. This is a consequence of  the properties of the spherical rearrangement and the fact that $u$ is radial with no inner jump, which implies
\begin{align*}
\Ece(\widetilde{u})&=\int_{\widetilde{\Omega}}\left(\frac{1}{2}|\psi'(|x|)|-\Theta(\psi(|x|))\right)d\Ln+\int_{\partial\widetilde{\Omega}}\left(\psi(|x|)^2+2c\psi(|x|)^{1+\epsilon}\right)d\Ls\\
&\leq\int_{\Omega}\left(\frac{1}{2}|\psi'(|x|)|-\Theta(\psi(|x|))\right)d\Ln+\int_{\partial^*\Omega}\left(\psi(|x|)^2+2c\psi(|x|)^{1+\epsilon}\right)d\Ls\\
&=\Ece(u).
\end{align*}
Where in the last line we used that $J_u=\partial^*\Omega$ up to a $\Ls$-negligible set, and either $\overline{u}$ or $\underline{u}$ is $0$ for $\Ls$-every point of $\partial\Omega$ (since there are no inner boundaries).\bigbreak

We know from the above procedure that $\widetilde{u}$ is locally radial in $0$ and in $p\neq 0$. Thus, for any point $x\in \widetilde{\Omega}\setminus (0,p)$, $\nabla u(x)$ is proportional to $x$ and $x-p$, which means that it is 0. This means that $\widetilde{u}$ is locally constant, which is not the case because of the equation $-\Delta u=(c+u)^{q-1}$ verified by $u$.\bigbreak

The consequence is that $|\widetilde{\Omega}\cap\{ x_1<0\}|=|\widetilde{\Omega}\cap\{ x_1>0\}|$ which proves that $\Omega$ is a radial set.\bigbreak

\medskip
\noindent {\bf Step 4.} {\bf The set $\Om$ is a ball.} We already know that $\Omega$ is a connected radial set. To show that $\Om$ is a ball, it suffices to exclude that it is an annulus. Suppose by contradiction $\Om=B_{r_2}\setminus \ov B_{r_1}$, where we can assume that $r_1>0$, since a point has zero capacity.


We know that $u$ is of the form $\psi(|x|)1_\Omega(x)$ where $\psi$ verifies \eqref{eq:ODE}, 
with the modified Robin boundary condition
\[ (-1)^i \psi'(r_i)+\beta\left(\psi(r_i)+c(1+\epsilon)\psi(r_i)^\epsilon\right)=0, \quad i=1,2.\]

We also know that $\psi> \delta>0$  on $\overline{\Omega}$. 
This means that $\psi$ can be extended slightly to a neighbourhood of $[r_1,r_2]$ as a solution of  \eqref{eq:ODE}, as long as $\psi>0$. We can thus extend $u$ as $\widetilde{u}(x)=\psi(|x|)$. Let:
\[g(\varrho_1,\varrho_2)=\Ece(\widetilde{u}1_{B_{\varrho_2}\setminus B_{\varrho_1}}),\]
where $\rho_1,\rho_2$ are taken close to $r_1,r_2$. We know that $g$ is minimal in $(r_1,r_2)$ among the set $\{ (\varrho_1,\varrho_2): \ \varrho_2^n-\varrho_1^n=r_2^n-r_1^n\}$. This implies:
\[\frac{\partial_{\varrho_1}g(r_1,r_2)}{r_1^{n-1}}+\frac{\partial_{\varrho_2}g(r_1,r_2)}{r_2^{n-1}}= 0.\]
This gives the following inequality, in which we shortened $\psi_i=\psi(r_i)$, $\psi_i'=\psi'(r_i)$

\begin{multline*}
\frac{1}{2}(|\psi_2'|^2-|\psi_1 '|^2)+(\Theta(\psi_1)-\Theta(\psi_2))+\frac{\beta}{2}(n-1)\left(\frac{\psi_1^2 + 2c\psi_1^{1+\epsilon}}{r_1}+\frac{\psi_2^2 + 2c\psi_2^{1+\epsilon}}{r_2}\right)\\
+\beta\left(\psi_1+c(1+\epsilon)\psi_1^\epsilon\right)\psi_1'+\beta\left(\psi_2+c(1+\epsilon)\psi_2^\epsilon\right)\psi_2 '=0.
\end{multline*}

We know that the (modified) Robin boundary condition is verified in  $r_1$ and $r_2$, which means that $\beta\left(\psi_i+c(1+\epsilon)\psi_i^\epsilon\right)=\pm\psi_i'$ ($+$ for $i=1$, $-$ for $i=2$). We rewrite the previous inequality as:
\begin{multline*}
\left(\frac{1}{2}|\psi_1'|^2+\frac{1}{q}(c+\psi_1)^{q}\right)-\left(\frac{1}{2}|\psi_2'|^2+\frac{1}{q}(c+\psi_2)^{q}\right)\\
+\frac{\beta}{2}(n-1)\left(\frac{\psi_1^2 + 2c\psi_1^{1+\epsilon}}{r_1}+\frac{\psi_2^2 + 2c\psi_2^{1+\epsilon}}{r_2}\right)=0.
\end{multline*}
In particular, this means 
$$
\frac{1}{2}|\psi_1'|^2+\frac{1}{q}(c+\psi_1)^{q}<\frac{1}{2}|\psi_2'|^2+\frac{1}{q}(c+\psi_2)^{q}.
$$
This is a contradiction because $\psi$ verifies \eqref{eq:ODE}, which implies:
\[\frac{d}{d\varrho}\left(\frac{1}{2}|\psi'|^2+\frac{1}{q}(c+\psi)^{q}\right)=-\frac{n-1}{r}|\psi'|^2\leq  0.\]
This means that a nontrivial annulus cannot be stationary for $\Ece$, which implies in particular that it cannot be a minimizer.
\end{proof}
\medskip

\subsection{Existence of minimizers}
Before proving the existence of minimizer, we need some technical results to localize the mass of a function in $SBV^{1/2}$. Below, we denote by $K_p$ the unit cube centered in $p\in\mathbb{Z}^n$.
\begin{lemma}\label{LemmaLocalization}
Let $u$ be a non-trivial positive function on $SBV^{1/2}(\mathbb{R}^n)$. There exists $p\in\mathbb{Z}^n$ such that
\[|\{ u>0\}\cap K_{p}|\geq \left(\frac{C\Vert u\Vert_{L^2}^2}{\Vert u\Vert_{L^2}^2+Q(u)}\right)^n, \]
where $C=C(n)>0$.
\end{lemma}

\begin{proof}

The family $ (K_p)_{p\in\mathbb{Z}}$ is a covering of $\mathbb{R}^n$. For all $p$, we have
\begin{multline*}
\Vert u\Vert_{L^2(K_p)}^2\leq |\{ u>0\}\cap K_p|^{1/n}\Vert u^2\Vert_{L^{\frac{n}{n-1}}(K_p)}\\
\leq C\left(\sup_q |\{ u>0\}\cap K_q|\right)^{1/n}\left(\Vert u\Vert_{L^2(K_p)}^2+Q_{|K_p}(u)\right) 
\end{multline*}
since $u^2_{|K_p}\in BV(K_p)$ for each $p$. Summing over $p$, we get, for possibly a different constant,
\[\Vert u\Vert_{L^2}^2\leq C\left(\sup_q |\{ u>0\}\cap K_q|\right)^{1/n}\left(\Vert u\Vert_{L^2}^2+Q(u)\right).\]
\end{proof}
Below, we give an obvious  variation on the compactness result of \cite{Buc10}.
\begin{lemma}\label{lemmaExtraction}
Let $(u_i)$ be a sequence of $SBV^{1/2}(\mathbb{R}^n)$ such that
\[\limsup_i\left( Q(u_i)+N^{c,\epsilon}(u_i)+\Vert u_i\Vert_{L^2}^2\right)<\infty.\]
Then there exists a subsequence, still denoted with the same index, and a function $u\in SBV^{1/2}(\mathbb{R}^n)$ such that
 $u_i\rightarrow u$ in $L^2_\text{loc}(\mathbb{R}^n)$ and
for all open set $A\Subset\mathbb{R}^n$, 
$$Q_{|A}(u)\leq \liminf_{i\rightarrow\infty} Q_{|A}(u_i)\qquad \mbox{ and }\qquad N_{|A}^\epsilon(u)\leq \liminf_{i\rightarrow\infty} N_{|A}^\epsilon(u_i).$$
\end{lemma}


 We are now in a position to prove Theorem \ref{bgn03}
\begin{proof}[ Proof of Theorem \ref{bgn03}]
 In view of the preceding results, it suffices to prove the existence of a minimizer.
 The proof goes as follows. Let $(u_i)$ be a minimizing sequence.
\begin{itemize}[label=\textbullet]
\item We show that up to a translation and up to subsequences, $(u_i)_i$ converges in some sense to a non-trivial function $u$ such that $|\{u>0\}|=:m'\in ]0,m]$.
\item We show that $u$ minimizes $\Ece$ in $\mathcal{U}_{m'}$.
\item We suppose by contradiction that $m'<m$: this allows us to find a sequence $(p_i)$ such that $|p_i|\rightarrow\infty$ and $(u_i(\cdot-p_i))_i$ converges to another non-trivial minimizer $v$ in $\mathcal{U}_{m''}$ for $m''\in ]0,m-m']$.
\item Then, knowing the structure of minimizers (smooth function defined on balls), we modify the sequence $(u_i)_i$ to build a new minimizing sequence that converges to a function whose support is given by the union of two disjoint balls, which is absurd. Thus $m'=m$ and $u\in \Um$ is the minimizer we were looking for.
\end{itemize}
\par
We proceed in several steps.

\vskip10pt\noindent \textbf{Step 1: Convergence of }$\mathbf{(u_i)}$.
We remind that, according to Lemma \ref{ResultControl}, a minimizing sequence $(u_i)_i$ verifies
\[\underset{i\geq 1}{\sup}\left(Q(u_i)+N^{c,\epsilon}(u_i)+\Vert \Theta (u_i)\Vert_{L^1}\right) <\infty.\]
Moreover, since for all $w\in\mathcal{U}_m$, $\Ece(tw)<0$ for small enough $t$, we can suppose without loss of generality that $-\Ece(u_i)$ is bounded away from $0$. In particular this implies that $\Vert\Theta (u_i)\Vert_{L^1}$ is bounded away from 0. This with the fact that $u\in\mathcal{U}_m$ implies that $\Vert u_i\Vert_{L^2}$ is bounded away from 0. We let
\[S_i:=\underset{p\in\mathbb{Z}^n}{\sup}\int_{K_p}u_i^2d\Ln.\]
Then $(S_i)$ is a bounded sequence and we can suppose that it converges up to subsequences to a limit $S\geq 0$. 
\par
We show that $S>0$.
Indeed, assume by contradiction that $S=0$. Let $K_{p_i}$ be a cube chosen by applying Lemma \ref{LemmaLocalization} to $u_i$, with $p_i\in \mathbb{Z}^n$. Since $Q(u_i)$ is bounded and $\Vert u_i\Vert_{L^2}$ is bounded below, then for a constant $\delta>0$ that does not depend on $i$
\[|K_{p_i}\cap\{u_i>0\}|\geq 2\delta.\]
Let 
$$
v_i:=u_i1_{\mathbb{R}^n \setminus \widetilde{K}_{p_i}},
$$
where $\widetilde{K}_{p_i}$ is a slightly smaller version of $K_{p_i}$ chosen so that the integral of $u_i^2$ on $\partial\widetilde{K}_{p_i}$ goes to 0 as $i\rightarrow \infty$ (using the assumption that $S=0$), and $|\widetilde{K}_{p_i}\cap\{u_i>0\}|\geq \delta$. Then, since $S=0$ 
\[\Ece(v_i)\leq \Ece(u_i)+o_{i\rightarrow\infty}(1)\]
so that $(v_i)$ is also a minimizing sequence. Since $|\{ v_i>0\}|\leq m-\delta$, letting $t:=\left(\frac{m}{m-\delta}\right)^{1/n}$, $v_i(\cdot/t)$ is in $\mathcal{U}_m$ and we get 
\[\underset{\mathcal{U}_m}{\inf}\Ece\leq \liminf_{i\rightarrow\infty}  \Ece\left(v_i(\cdot/t)\right)\leq t^{n}\liminf_{i\rightarrow\infty} \Ece(v_i)=t^n \underset{\mathcal{U}_m}{\inf}\Ece,\]
which is absurd (recall that $\underset{\mathcal{U}_m}{\inf}\Ece<0$).
Hence, for all $i$ large enough there exists $p_i\in\mathbb{Z}^n$ such that
\begin{equation}
\label{eq:Si}
\int_{K_{p_i}}u_i^2d\Ln\geq \frac{S}{2}>0.
\end{equation}
We translate each $u_i$ in $\RR^n$ and assume that $p_i=0$ for all $i$. Using Lemma \ref{lemmaExtraction}, there exists $u\in SBV^{1/2}(\mathbb R^n)$ such that
\begin{align*}
u_i\underset{L^2_\text{loc}}{\longrightarrow} &u \not=0\\
\end{align*}
with local lower semicontinuity on $Q$ and $N^{c,\epsilon}$.

\vskip10pt\noindent \textbf{ Step 2: The limit of $\mathbf{(u_i)_i}$ is a minimizer of $\Ece$ in $\mathcal{U}_{m'}$ where $m':=|\{ u>0\}|\in ]0,m]$. }  By Step 1 we know that $u$ is nontrivial. 
Let $v\in \mathcal{U}_{m'}$, and let us consider the functions in $SBV^{1/2}(\mathbb R^n)$ given by
\[
v^r=1_{B_r}v+1_{\mathbb{R}^n\setminus B_r}u 
\qquad\mbox{ and } \qquad
v_i^r=1_{B_r}v+1_{\mathbb{R}^n\setminus B_r}u_i.
\]

We claim that we can find a set $D\subseteq \mathbb R$ with $|D|=0$ such that for $r\not\in D$
\begin{equation}
\label{eq:claim2}
\Ece(u)-\Ece(v^{r})\leq C\liminf_{i\to \infty}\left(1-\frac{m}{m_i ^{r}}\right)_+, 
\end{equation}
and
\begin{equation}
\label{eq:claim3}
\liminf_{i\rightarrow \infty}(m_i^{r}-m)\le e_r\underset{r\rightarrow \infty}{\longrightarrow} 0,
\end{equation}
where $m_i^r:=|\{v_i^r>0\}|$, and $C$ does not depend on $i$ and $k$. 
\par
We can then find $r_k\to +\infty$ with $r_k\not\in D$ such that
\begin{equation}
\label{eq:claim1}
\Ece(v^{r_k})\underset{k\rightarrow \infty}{\longrightarrow}\Ece(v).
\end{equation}
 Indeed we have 
$$
J_{v^r}=(J_u\cap(\RR^n\setminus\overline{B_r}))\cup (J_v\cap B_r)\cup J^r,
$$ 
where $J^r$ is the subset of points in $\partial B_r$ where $\overline{u1_{\RR^n\setminus B_r}}\neq \overline{v1_{B_r}}$. 
Choose $r$ such that $\Ls(\partial B_r\cap J_u)=\Ls(\partial B_r\cap J_v)=0$, which amounts to choosing almost any $r>0$. Then

\[\Ece(v^r)-\Ece(v)=\mathcal{E}^{c,\epsilon}(u1_{|\mathbb{R}^n\setminus B_r})-\mathcal{E}^{c,\epsilon}(v1_{|\mathbb{R}^n\setminus B_r})+S^r\]
where
\begin{equation}
\label{eq:Sr}
S^r=\frac{\beta}{2}\int_{J^r}\left[(u^2+2cu^{1+\epsilon}) +(v^2+2cv^{1+\epsilon})\right]d\Ls.
\end{equation}
$\mathcal{E}^{c,\epsilon}(u1_{|\mathbb{R}^n\setminus B_r})-\mathcal{E}^{c,\epsilon}(v1_{|\mathbb{R}^n\setminus B_r})$ goes to zero as $r$ goes to infinity by dominated convergence. Since $u$ and $v$ belong to $L^2$ with supports of volume $m'$, they belong also to $L^{1+\epsilon}$, which means that $\int_0^\infty S^r dr<\infty$; in particular this implies that we can find $r_k\to +\infty$ with $r_k\not\in D$ such that $S^{r_k}\rightarrow 0$, so that \eqref{eq:claim1} follows.
\par
Taking \eqref{eq:claim1}, \eqref{eq:claim2} and \eqref{eq:claim3} into account, we infer easily that 
$$
\Ece(u) \le \Ece(v),
$$ 
i.e. $u$ is a minimizer of $\Ece$ in $\mathcal{U}_{m'}$. In particular, $u$ is supported on a ball.
\par
In order to conclude Step 2, we prove claims \eqref{eq:claim2} and \eqref{eq:claim3}.
\begin{itemize}
%
\item[(a)] Let us consider claim \eqref{eq:claim2}. Comparing $v^r$ and $u$ we get   
\[\Ece(u)-\Ece(v^r)=\mathcal{E}^{c,\epsilon}(u1_{|B_r})-\mathcal{E}^{c,\epsilon}(v1_{|B_r})-S^r,\]
 where $S^r$ is given by \eqref{eq:Sr}.
Setting 
$$
S^r_i=\frac{\beta}{2}\int_{\partial B_r\cap \{ u\neq v\}^1}\left[(u_i^2+2cu_i^{1+\epsilon}) +(v^2+2cv^{1+\epsilon})\right]d\Ls
$$ 

we have for a.e. $r>0$
$$
S_i^r\underset{i\rightarrow \infty}{\longrightarrow} S^r
$$
as $u_i\underset{L^2(\partial B_r)}{\rightarrow} u$ for a.e. $r>0$. We infer

\begin{align}
\label{eq:ineq}
\Ece(u)-\Ece(v^r)&\leq \liminf_{i\rightarrow\infty}\left(\mathcal{E}^{c,\epsilon}_{|B_r}(u_i)-\mathcal{E}^{c,\epsilon}_{|B_r}(v)\right)-S^r\\
\nonumber&=\liminf_{i\rightarrow\infty}\left(\Ece(u_i)-\Ece(v_i^r)+S_i^r\right)-S^r\\
\nonumber&=\liminf_{i\rightarrow\infty}\left(\Ece(u_i)-\Ece(v_i^r)\right) \\
\nonumber&=\inf_{\Um}\Ece-\limsup_{i\rightarrow\infty}\Ece(v_i^r).
\end{align}
\par
If $m_i^r\leq m$, we know this last quantity is nonpositive. Suppose now that $m_i^r> m$, and let
\[w_i^r(x):=v_i^r\left(\Big[\frac{m_i^r}{m}\Big]^{1/n}x\right).\]
Then $w_i^r$ has a support of volume $m$, and so with a change of variable
\begin{align*}
 \inf_{\Um}\Ece \le  \Ece(w_i^r)&\leq \Ece(v_i^r)+\left(1-\frac{m}{m_i ^r}\right)\Vert\Theta(v_i^r)\Vert_{L^1}\\
&\leq \Ece(v_i^r)+\left(1-\frac{m}{m_i ^r}\right)\Vert\Theta(v)+\Theta( u_i )\Vert_{L^1}\\
&\leq \Ece(v_i^r)+C\left(1-\frac{m}{m_i ^r}\right), 
\end{align*}
 where $C$ does not depend on $i,r$.  Then 
 $$
\inf_{\Um}\Ece\le \limsup_{i\to \infty}\Ece(v_i^r)+ C\liminf_{i\to \infty}\left(1-\frac{m}{m_i ^r}\right),
$$
so that coming back to \eqref{eq:ineq} we deduce
$$
\Ece(u)-\Ece(v^r)\le C\liminf_{i\to \infty}\left(1-\frac{m}{m_i ^r}\right).
$$
Collecting both cases $m_i^r\leq m$ and $m_i^r>m$, we get precisely claim \eqref{eq:claim2}.
\vskip5pt
\item[(b)] Let us come to claim \eqref{eq:claim3}. We have for every $r>0$

\begin{align*}
\liminf_{i\rightarrow \infty}(m_i^r-m)&=\liminf_{i\rightarrow \infty}\left(|\{ v>0\}\cap B_r|-|\{ u_i>0\}\cap B_r|\right)\\
&=|\{ v>0\}\cap B_r|-\limsup_{i\rightarrow \infty}\left(|\{ u_i>0\}\cap B_r|\right)\\
&\le |\{ v>0\}\cap B_r|-|\{ u>0\}\cap B_r|=:e_r\underset{r\rightarrow\infty}{\longrightarrow}0.
\end{align*}
\end{itemize}

\vskip10pt\noindent \textbf{Step 3: Existence of another sequence $\mathbf{(u_i)_i}$ such that $\mathbf{(u_i(\cdot-p_i))_i}$ converges if $m'<m$}. 
 Let $R>0$ be such that the support of $u$ is contained in $B_R$, so that
\begin{equation}
\label{eq:uiL2}
u_i\underset{L^2_{loc}(\mathbb R^n\setminus B_R)}{\longrightarrow} 0.
\end{equation}
We can suppose that 
$$
\Vert u_i\Vert_{L^2(\partial B_R)}\rightarrow 0.
$$ 
Let 
$$
v_i:=u_i 1_{\mathbb{R}^n\setminus B_{2R}}.
$$ 
 Notice that
\begin{equation}
\label{eq:bound-vi}
Q(v_i)+N^{c,\epsilon}(v_i)\leq 
C.
\end{equation}
We claim 
\begin{equation}
\label{eq:liminf}
\liminf_{i\to \infty}\Vert v_i\Vert_{L^2}^2>0
\end{equation}

Indeed by contradiction, we would have (up to a subsequence) $\Vert v_i\Vert_{L^2(\mathbb{R}^n)}\rightarrow 0$, which implies strong $L^2$ convergence of $(u_i)$ to $u$ and so $\Vert \Theta(u_i)\Vert_{L^1}\rightarrow \Vert \Theta(u)\Vert_{L^1}$. This implies by lower semicontinuity of $Q$ and $N^{c,\epsilon}$
\[\Ece(u)\leq \liminf_{i\rightarrow\infty}\Ece(u_i)=\inf_{\Um}\Ece.\]
But since $m'<m$, then for $t:=\Big[\frac{m}{m'}\Big]^{1/n}>1$, we get
\[\Ece(u(\cdot/t))<t^n\inf_{\Um}\Ece\]
which is absurd since $u(\cdot/t)$ is in $\Um$ (recall $\inf_{\Um}\Ece<0$). Then \eqref{eq:liminf} follows. 
\par
This gives, in particular, that
\[\liminf_{i\rightarrow\infty} \left(\frac{c\Vert v_i\Vert_{L^2}^2}{\Vert v_i\Vert_{L^2}^2+Q(v_i)}\right)^n >0.\]
Taking into account \eqref{eq:bound-vi}, we can then proceed as in Step 2 (by defining the ``concentration" $S$) to show that for each $i$ there exists $p_i\in\mathbb{R}^n$ such that 
\[\liminf_{i\rightarrow\infty}\left(\Vert v_i\Vert_{L^2(K_{p_i})}\right)>0.\]
Notice that in view of \eqref{eq:uiL2} we know that necessarily 
\begin{equation}
\label{eq:pi}
|p_i|\rightarrow+\infty.
\end{equation}
As previously, we can show that $u_i(\cdot-p_i)$ converges in the $L^2_\text{loc}$ sense to a function $v$ in $\mathcal{U}_{m''}$, where $m''\in ]0,m-m']$ that is a minimizer of $\Ece$ in $\mathcal{U}_{m''}$. As a consequence, its support is also a ball.\bigbreak

\vskip10pt\noindent \textbf{Step 4: Construction of an incompatible minimizing sequence in the case $m'<m$}. The idea is now to bring together $u$ and $v$ constructed in Step 2 and 3. While each ball is optimal, the union of two disjoint balls is not (because it is not a ball, which is the only possible minimizer). Let $R>0$ be big enough such that the supports of $u$ and $v$ are contained in $B_R$. Let 
$$
w_i(x):=(1_{B_{R}(p_i)} u_i)(p_i+x).
$$
We can choose $R$ slightly bigger such that $\Vert w_i\Vert_{L^2(\partial B_{R}(p_i))}\rightarrow 0$. Let
\[
\widetilde{u_i}(x):=1_{\mathbb{R}^n\setminus (B_{R}(2Re_n)\cup B_{R}(p_i))}(x)u_i(x)+w_i(x-2Re_n).
\]
Then, up to choosing a slightly bigger $R$ (so that the boundary terms all go to 0), we get
\[\Ece(\widetilde{u_i})\leq \Ece(u_i)+o_{i\rightarrow \infty}(1).\]
By construction $|\{\widetilde{u_i}>0\}|\leq m$, so up to a dilation that can only decrease $\Ece$, $(\widetilde{u_i})$ is also a minimizing sequence.

By similar arguments, we know $\widetilde{u_i}$ converges to a minimizer $\widetilde{u}$ for a certain volume less than $m$. But by construction, the support of this limit is exactly the union of two balls, so it cannot be a minimizer. This means that our assumption $m'<m$ is false.

The consequence is that $m'=m$, and so that $u$ is a minimizer of $\Ece$ in $\Um$.
\end{proof}


\section{The ball minimizes the geometric functional}\label{bgn12}
\begin{proposition}
\label{prop:minballc}
Let $u\in \Um$, then
\[\Ec(u)\geq E^c(B^m).\]
\end{proposition}

\begin{proof} Let $u_m^{c,\epsilon}$ be a minimizer obtained in Theorem \ref{bgn03} whose support is $B^m$. Then

\[E^c(B^m)\leq\Ec(u_m^{c,\epsilon}).\]
Moreover:
\[\Ec(u_m^{c,\epsilon})-\Ece(u_m^{c,\epsilon})=c\beta\int_{\partial B^m}\Big[u_m^{c,\epsilon}-(u_m^{c,\epsilon})^{1+\epsilon}\Big]d\Ls.\]

A sufficient condition for this term to go to 0 when $\epsilon$ goes to 0 is
\[\limsup_{\epsilon\rightarrow 0}\Vert u_m^{c,\epsilon}\Vert_{L^\infty(\partial B^m)}<\infty.\]
Since $u_m^{c,\epsilon}$ is radial,
\[C(n,\beta,c,m)= \Ece(1_{B^m})\geq \Ece(u_m^{c,\epsilon})=\Big[Q(u_m^{c,\epsilon})-\int_{B^m}\Theta(u_m^{c,\epsilon})d\Ls\Big]+N^{c,\epsilon}(u_m^{c,\epsilon}).\]
The first term is bounded below uniformly in $\epsilon$ due to inequality \eqref{bgn16} and
\[N^{c,\epsilon}(u_m^{c,\epsilon})=c\beta \Ls\left(\partial B^m\right) \Vert u_m^{c,\epsilon}\Vert_{L^\infty(\partial B^m)}^{1+\epsilon},\]
 so $ \Vert u_m^{c,\epsilon}\Vert_{L^\infty(\partial B^m)}$ is bounded uniformly as $\epsilon\rightarrow 0$.

Take $u$ in $\Um$.  We know from Theorem \ref{bgn03} that
\[\Ece(u)\geq \Ece(u_m^{c,\epsilon})\]
The previous discussion gives
\[\liminf_{\epsilon\rightarrow 0}\Ece(u)\geq E^c(B^m).\]
In order to conclude, we may assume $\Ec(u)<+\infty$ and it suffices to check that 
$$
\Ece(u)\goto\Ec(u),
$$
which amounts to show
\[\lim_{\epsilon\rightarrow 0}\int_{J_u}(\overline{u}^{1+\epsilon}+\underline{u}^{1+\epsilon})d\Ls= \int_{J_u}(\overline{u}+\underline{u})d\Ls.\]
Since $u^{1+\epsilon}\leq u+u^2$ for all $u$ and all $\epsilon\in ]0,1[$, the dominated convergence leads to the result.
\end{proof}

\section{Proof of Theorems \ref{mainresult} and \ref{mainresult.2}}\label{bgn09}

\medskip
\noindent{\bf Proof of Theorem \ref{mainresult}.}
 Recall that 
$$
\mathcal E(u)=\frac{1}{2}\int_{\Rn}|\nabla u|^2 d\Ln+\frac{\beta }{2}\int_{J_v}\left[\underline{u}^2+\overline{u}^2\right]d\Ls -\frac{1}{q}\int_{\Rn}u^q\,dx.
$$

We shall prove the following stronger result. 

\begin{theorem}\label{mainresult.1}
Let $u\in \Um$. Then
\[
\mathcal{E}(u)- E(B^m)\geq \frac{\beta}{2}\left(\underset{\{ u>0\}}{\text{essinf}}\ u\right)^2\left[\int_{J_u}\left(1_{\{\overline{u}>0\}}+1_{\{\underline{u}>0\}}\right)d\Ls-\Ls(\partial B^m)\right].
\]
\end{theorem}

\begin{remark}
\label{}
{\rm
Note that $\int_{J_u}\left(1_{\{\overline{u}>0\}}+1_{\{\underline{u}>0\}}\right)d\Ls$ is larger than $\Ls(J_u)$, since it counts twice the points in $J_u$ on which both $\overline{u}$ and $\underline{u}$ are positive.
}
\end{remark}

\begin{proof}
 We may assume $\underset{\{ u>0\}}{\text{essinf}}\,u>0$. Let $c\in ]0,\underset{\{ u>0\}}{\text{essinf}}\ u[$. Then $(u-c)_+$ belongs to $\Um$ and Proposition \ref{prop:minballc} gives
\[\Ec ((u-c)_+)\geq E^c(B^m).\]
Moreover
\[\Ec((u-c)_+)=\mathcal{E}(u)-\frac{\beta}{2}c^2\int_{J_u}\left(1_{\{\overline{u}>0\}}+1_{\{\underline{u}>0\}}\right)d\Ls+\frac{c^q}{q}m\]
and
\[E^c(B^m)\geq E(B^m)-\frac{\beta}{2}c^2\Ls(B^m)+\frac{c^q}{q}m,\]
 so that the result follows. This proves as well Theorem \ref{mainresult} taking $u$ to be the minimizer of $E(\cdot;\Omega)$, extended by $0$ on $\RR^n \sm\Omega$.
\end{proof}

\medskip
\noindent{\bf Proof of Theorem \ref{mainresult.2}.}
It is enough to show:
\begin{equation}\label{EqQuantitativeE}
E(\Omega)-E(B)\geq C\mathcal{A}(\Omega)^2,
\end{equation}
where $C>0$ depends on $n,\beta,q$ and $|\Omega|$, only. Indeed we know that $E(\Omega)=\frac{q-2}{2q} \lb_q(\Omega)^\frac{q}{q-2}$, so
\[C\mathcal{A}(\Omega)^2\leq \frac{q-2}{2q} \lb_q(\Omega)^\frac{q}{q-2}-\frac{q-2}{2q} \lb_q(B)^\frac{q}{q-2}.\]
 Since $\mathcal{A}(\Omega)\le 2$, up to replacing $C$ with   a smaller constant, we loose no generality by supposing that $\lb_q(\Omega)$ is close to $\lb_q(B)$. For $\lb_q(\Omega)$ close enough to $\lb_q(B)$, we have:
\[\frac{q-2}{2q} \lb_q(\Omega)^\frac{q}{q-2}-\frac{q-2}{2q} \lb_q(B)^\frac{q}{q-2}\leq \lb_q(B)^\frac{2}{q-2}\left(\lb_q(\Omega)-\lb_q(B)\right),\]
which proves the result. Let us now show \eqref{EqQuantitativeE}.\bigbreak

From now on, we let $m:=|\Omega|$. The idea is to use an intermediate set $A$ such that $\mathcal{A}(A)$ and $\mathcal{A}(\Omega)$ are of the same order, but $\inf\ u_A\geq C(n,\beta,|\Omega|)$ for a certain positive constant. This may be done by considering the following auxiliary minimisation problem
\begin{equation}\label{OptProb}
\inf\{ E(A)+k|A|,\ A\subset\Omega\},
\end{equation}
for a suitable constant $k>0$ (that may be chosen such that if we restrict the admissible sets to balls of volume less than $m$, the optimal ball has the measure equal to $m$). We need to check such a minimizer exists. Showing the existence and regularity of such a minimizer is similar to the work that has been done in Section \ref{bgn04}, we outline the main steps.

\vskip10pt\noindent {\bf Step 1.} We introduce the following relaxed version in $SBV^\frac12$
\[\inf\{ \mathcal{E}_{k}(u),\ |\{ u>0\}\setminus\Omega|=0 \},\]
where $\mathcal{E}_{k}(u)=\mathcal{E}(u)+k|\{ u>0\}|$.

\vskip10pt\noindent {\bf Step 2.} We show the existence of a minimizer of the relaxed problem. 
Take  a minimizing sequence $(u_i)$, up to  subsequences  we can suppose it converges in $L^2_\text{loc}(\mathbb{R}^n)$ to a function $u$ with support in $\Omega$. Since  $Q$ is lower semi-continuous, we only need to show that $(u_i)$ is tight in the sense that
\[\underset{R\rightarrow\infty}{\limsup}\ \underset{i\geq 1}{\sup}\left(\Vert u_i\Vert_{L^2(\Omega\setminus B_R)}\right)=0.\]
With the $L^2_{\text{loc}}$ convergence, this will prove that $\Vert u_i\Vert_{L^q(\mathbb{R}^n)}\underset{i\rightarrow \infty}{\longrightarrow}\Vert u\Vert_{L^q(\mathbb{R}^n)}$. Let $\chi_R$ be a smooth function such that
\[\begin{cases}
0\leq \chi_R\leq 1\\
\chi_R\equiv 1\text{ in }\mathbb{R}^n\setminus B_{R}\\
\chi_R\equiv 0\text{ in }B_{R-1}\\
\Vert\nabla \chi_R\Vert_{L^\infty(\mathbb{R}^n)}\leq 2,
\end{cases}\]
and let $v_{i,R}=\chi_R u_i$. Then:
\begin{align*}
\Vert u_i\Vert_{L^2(\mathbb{R}^n\setminus B_R)}^2&\leq \Vert v_i\Vert_{L^2(\mathbb{R}^n\setminus B_R)}^2\\
&\leq \lambda_2\left(B^{|\{ v_i>0\}|}\right)^{-1}Q(v_i)& \text{using inequality  \eqref{bgn16} for $q=2$}\\
&\leq \lambda_2\left(B^{|\Omega\setminus B_{R-1}|}\right)^{-1}(2Q(u_i)+8\Vert u_i\Vert_{L^2(\mathbb{R}^n)}^2)&\text{since }|\nabla v_i|^2\leq 2|\nabla u_i|^2+8u_i^2\\
&\leq C\lambda_2\left(B^{|\Omega\setminus B_{R-1}|}\right)^{-1}Q(u_i)&\text{where }C\text{ depends on }n,m.
\end{align*}
$Q(u_i)$ is bounded uniformly in $i$ and $\lambda_2\left(B^{|\Omega\setminus B_{R-1}|}\right)^{-1}$ goes to $0$ when $R\rightarrow\infty$: this proves what we wanted to show.
\par
Thus we know that
\[\mathcal{E}_{k}(u)\leq \underset{i\rightarrow\infty}{\liminf}\ \mathcal{E}_{k}(u_i)\]
which means $u$ is a minimizer.

\medskip
\noindent {\bf Step 3.} We show that $u\geq c 1_{\{u>0\}}$ for a certain $c=c(n,\beta,k,E(\Omega))>0$; this is essentially the same proof as Lemma \ref{lemmaCaff} in the case $\epsilon=1$. We then show that $u$ is bounded and defines an open domain $A$ the same way we did  in Lemma \ref{bounded} and Lemma \ref{lem:support}. 

\vskip10pt\noindent{\bf  Step 4. }
 The map  $\rho\mapsto E(B_\rho)$ is a strictly decreasing function since for any ball $B'\Subset B$, one has:
\[E(B)\leq\frac{|B|}{|B'|}E(B'),\]
which implies 
\[\frac{d}{d\rho}E(B_\rho)\leq \frac{nE(B_\rho)}{\rho}(<0).\]
 We deduce that for any $k>0$ small enough depending only on the parameters $(n, q,\beta,m)$, the function
\begin{equation}\label{Decreasing}
\begin{cases} [0,r_m]\rightarrow \mathbb{R}\\
\rho\mapsto E(B_\rho)+2k|B_\rho|\end{cases}
\end{equation}
admits a minimum in $r_m$ (where $r_m$ is such that $|B_{r_m}|=m$). 

\vskip10pt\noindent{\bf  Step 5 }
As remarked at the beginning of the proof, we can suppose without loss of generality that $E(\Omega)\in [E(B),\frac{1}{2}E(B)]$. Letting $k$ be given by Step 4, the solution $u$ of the relaxed problem \eqref{OptProb} is thus such that  $u\geq c$ on its support $A$, where $c$
depends only on the parameters $(n, q, \beta,m)$.
\par
We can then write (with $c$ being a constant only depending on $(n, q , \beta,m)$ that may not be the same from line to line):
\begin{align*}
E(\Omega)&\geq E(A)-k|\Omega\setminus A|&\text{ because }A\text{ solves \eqref{OptProb},}\\
&\geq E(B^{|A|})-k|\Omega\setminus A|+c\Big[\Ls(\partial A)-\Ls(\partial B^{|A|})\Big]&\text{  in view of Theorem }\ref{mainresult},\\
&\geq E(B^{|A|})-k|\Omega\setminus A|+c|A\triangle B^{|A|}|^2&\text{ using the inequality \eqref{bgn01},}\\
&\geq E(B^{|\Omega|})+k|\Omega\setminus A|+c|A\triangle B^{|A|}|^2&\text{ using that \eqref{Decreasing} has a minimum in $r_m$,}\\
&\geq E(B^{|\Omega|})+c\Big[2|\Omega\setminus A|+|A\triangle B^{|A|}|\Big]^2&\text{ because }|\Omega\setminus A|\geq |\Omega|^{-1}|\Omega\setminus A|^2,\\
&\geq E(B^{|\Omega|})+c\mathcal{A}(\Omega)^2&\text{ we detail that point below,}
\end{align*}
so that Theorem \ref{mainresult.2} is proved.
\par
We detail the last inequality. Let $B^{|\Omega|}$ be any ball containing $B^{|A|}$. Then:
\begin{align*}
2|\Omega\setminus A|+|A\triangle B^{|A|}|&=\left(|\Omega\setminus A|+|A\setminus B^{|A|}|\right)+\left(|\Omega\setminus A|+|B^{|A|}\setminus A|\right).
\end{align*}
On one hand
\begin{align*}
|\Omega\setminus A|+|A\setminus B^{|A|}|&\geq |\Omega\setminus B^{|A|}|\geq |\Omega\setminus B^{|\Omega|}|.\\
\end{align*}
On the other hand
$$
|\Omega\setminus A|+|B^{|A|}\setminus A|= |\Omega|-|A|+|B^{|A|}|-|B^{|A|}\cap A|$$
$$
= |B^{|\Omega|}|-|B^{|A|}\cap A|= |B^{|\Omega|}\setminus (A\cap B^{|A|})| \geq |B^{|\Omega|}\setminus \Omega|.
$$
{\bf Acknowledgments.} The first and the third authors were supported by the LabEx PERSYVAL-Lab GeoSpec (ANR-11-LABX-0025-01) and ANR SHAPO (ANR-18-CE40-0013). 
\bibliographystyle{plain}
\bibliography{RefSaintVenant}

\end{document}